\documentclass[a4paper]{amsart}
\usepackage{amssymb}

\input xy
\xyoption{all}

\newcommand{\Id}{\mathrm{Id}}
\newcommand{\Comp}{\mathsf{Comp}}

\newcommand{\pr}{\mathrm{pr}}

\newcommand{\conv}{\mathrm{conv}}

\newcommand{\id}{\mathrm{id}}
\newcommand{\A}{\mathbb A}

\newcommand{\D}{\mathbb D}

\newcommand{\R}{\mathbb R}
\newcommand{\U}{\mathbb U}
\newcommand{\N}{\mathbb N}
\newcommand{\T}{\mathcal T}
\newcommand{\F}{\mathcal F}
\newcommand{\C}{\mathcal C}

\newcommand{\I}{\mathbb I}
\newcommand{\E}{\mathbb E}

\newtheorem{theorem}{Theorem}
\newtheorem{df}{Definition}

\newtheorem{problem}{Problem}
\newtheorem{proposition}{Proposition}

\begin{document}

\title{An isomorphism of idempotent monads }


\author{ Taras Radul}

\maketitle

Institute of Mathematics, Casimir the Great University of Bydgoszcz, Poland;
\newline
Department of Mechanics and Mathematics, Ivan Franko National University of Lviv,
Universytetska st., 1. 79000 Lviv, Ukraine.
\newline
e-mail: tarasradul@yahoo.co.uk



\begin{abstract} We consider isomorphism between the idempotent measure monad based on the maximum and the addition operations and the
idempotent measure monad based on the maximum and the multiplication  operations. A one of the consequences of this result is the
construction of a fuzzy integral based on the maximum and the addition operation.  We also investigate convexities related to these
monads.
 \end{abstract}

\maketitle

\section{Introduction} Idempotent mathematics has been a rapidly growing field in the last few decades. In traditional mathematics,
addition is a fundamental operation. However, in idempotent mathematics, a different operation, typically the maximum operation,
replaces addition.  Another classical multiplication operation was changed by different operations: addition, minimum, t-norm etc,
depending on the specific model needed for a given application. These changes have led to numerous applications in various fields,
including mathematics, mathematical physics, computer science, and economics. The versatility of idempotent mathematics is evident
from the wide range of problems it can address. For a comprehensive overview of this  field and its applications, readers may refer to
the survey article by Litvinov \cite{Litv} and the extensive bibliography provided therein.

Such fundamental  mathematical notion as convexity also found its idempotent analogues.  Max-plus convex sets were introduced in
\cite{Z}.
Max-min convexity was studied in \cite{NS} and \cite{NS1}. The B-convexity based on the operations of the maximum and the
multiplication was studied in \cite{BCh}.

Let us remark that many known abstract convexity structures have categorical nature and are  generated by monads \cite{R1}.
Topological and categorical properties of the functor of max-plus idempotent  measures  were studied in \cite{Zar}. In particular, the
idempotent measure monad was constructed. The monad of $\cdot$-measures (functionals which preserve the maximum and the multiplication
operations) was introduced in  \cite{R6}.  The convexity generated by this monads coincide with the convexities based on the maximum
and the multiplication  operations and were described in \cite{R2} in more general context.

We establish an isomorphism between the monad based on the maximum and addition operations and  the monad based on the maximum and
multiplication operations in this paper. Although such isomorphism is based on the simple fundamental fact that the exponent  function
transform addition into multiplication, it is technically difficult to build a direct isomorphism,  so we use correspondence between
measures and its densities. We also consider convexities generated by both monads which also are isomorphic in the last section of our
paper.

Composing the main result of this paper with the results from \cite{R6} we obtain that the the monad based on the maximum and addition
operations is isomorphic to the possibility capacity monad. We can consider such correspondence between measures as functions defined
on some family of sets and measures as functionals as a fuzzy integral related possibility capacities.   In fact, many applications of
capacities (non-additive measures) to game theory, decision making theory, economics etc deal not with measures as set functions  but
with integrals which allow to obtain expected utility or expected pay-off (see for example \cite{E}, \cite{BCh}, \cite{KZ}, \cite{EK}
\cite{Rad}, \cite{R3}, \cite{R4}).  Several types of integrals with respect to non-additive measures were developed for different
purposes (see for example books \cite{Grab} and  \cite{Den}). Such integrals are called fuzzy integrals. One of the important problems
of the fuzzy integrals theory is characterization of integrals as functionals on some function space (see for example subchapter 4.8
in \cite{Grab} devoted to characterizations of the Choquet integral and the Sugeno integral). Characterizations of t-normed integrals
were discussed in \cite{CLM}, \cite{Rad} and \cite{R5}. In fact these theorems we can consider as non-additive and non-linear
analogues of well-known  Riesz Theorem about a correspondence between the set of $\sigma$-additive regular Borel measures and the set
of linear  positively defined functionals. We introduce in this paper a fuzzy integral based on the maximum and the addition operation
and give it characterization.

\section{Idempotent measures and densities} In what follows, all spaces are assumed to be compacta (compact Hausdorff space) except
for $\R$ and the spaces of continuous functions on a compactum. All maps are assumed to be continuous.  We  denote by $C(X)$ the
Banach space of continuous functions on a compactum  $X$ endowed with the sup-norm.  We also consider the natural lattice operations
$\vee$ and $\wedge$ on $C(X)$ and  its sublattices $C(X,[0,+\infty))$ and $C(X,[0,1])$. For any $c\in\R$ we shall denote the
constant function on $X$ taking the value $c$ by $c_X$.

\begin{df}\cite{Zar} A functional $\mu: C(X) \to \R$ is called an idempotent probability  measure (a Maslov measure) if

\begin{enumerate}
\item $\mu(1_X)=1$;
\item $\mu(\lambda_X+\varphi)=\lambda+\mu(\varphi)$ for each $\lambda\in\R$ and $\varphi\in C(X)$;
\item $\mu(\psi\vee\varphi)=\mu(\psi)\vee\mu(\varphi)$ for each $\psi$, $\varphi\in C(X)$.
\end{enumerate}

\end{df}

Let $IX$ denote the set of all idempotent  measures on a compactum $X$. We consider topologically
$IX$ as a subspace of $\R^{C(X)}$.  The construction $I$ is  functorial what means that for each continuous map $f:X\to Y$ we can
consider a continuous map $If:IX\to IY$ defined as follows $If(\mu)(\psi)=\mu(\psi\circ f)$ for $\mu\in IX$ and $\psi\in C(Y)$.
Moreover $I$ preserves identities and composition. It is proved in \cite{Zar} that the functor $I$ preserves topological embedding.
For an embedding $i:A\to X$ we will identify the space $IA$ and the subspace $Ii(IA)\subset IX$.

By $\delta_{x}$ we denote the Dirac measure supported by the point $x\in X$.  It was shown in \cite{Zar} that the set $$I_\omega
X=\{\max_{i=1}^n(\lambda_i+\delta_{x_i})\mid\lambda_i\in\R_{\max},\ i\in\{1,\dots,n\},\ \max_{i=1}^n\lambda_i=0,\ x_i\in X,\
n\in\N\},$$ (i.e., the set of idempotent probability measures of finite support) is dense in $IX$.

By $\Comp$ we denote the category of compact Hausdorff
spaces (compacta) and continuous maps. We recall the notion  of monad (or triple) in the sense of S.Eilenberg and J.Moore \cite{EM}.
We define it only for the category $\Comp$.

A {\it monad} $\E=(E,\eta,\mu)$ in the category $\Comp$ consists of an endofunctor $E:{\Comp}\to{\Comp}$ and natural transformations
$\eta:\Id_{\Comp}\to E$ (unity), $\mu:E^2\to E$ (multiplication) satisfying the relations $$\mu\circ E\eta=\mu\circ\eta E=\text{\bf
1}_E$$ and $$\mu\circ\mu E=\mu\circ E\mu.$$ (By $\Id_{\Comp}$ we denote the identity functor on the category ${\Comp}$ and $E^2$ is
the superposition $E\circ E$ of $E$.)

The functor $I$ was completed to the monad $\I=(I,\eta,\mu)$  in \cite{Zar}. Let us describe the components of the natural
transformations $\eta$ and $\mu$. For a map $\phi\in C(X)$ we denote by $\pi_\phi$ or $\pi(\phi)$ the
corresponding projection $\pi_\phi:IX\to \R$. For a compactum $X$ we define components $\eta X$ and $\mu X$ of natural transformations
$\eta:\Id_{\Comp}\to I$, $\mu:I^2\to I$ by $\pi_\phi\circ \eta X=\phi$ and $\pi_\phi\circ \mu X=\pi(\pi_\phi)$ for all $\phi\in C(X)$.
It was proved in \cite{Zar} that the triple $\I=(I,\eta,\mu)$ forms a monad.

\begin{df}\cite{R6} A functional $\mu: C(X,[0,1]) \to [0,1]$ is called an $\cdot$ -measure  if

\begin{enumerate}
\item $\mu(1_X)=1$;
\item $\mu(\lambda_X\cdot\varphi)=\lambda\cdot\mu(\varphi)$ for each $\lambda\in[0,1]$ and $\varphi\in C(X,[0,1])$;
\item $\mu(\psi\vee\varphi)=\mu(\psi)\vee\mu(\varphi)$ for each $\psi$, $\varphi\in C(X)$.
\end{enumerate}

\end{df}

Let us denote that the term $\cdot$ -measure is not used in \cite{R6}. It is a particular case of the term $\ast$-measure from
\cite{Sukh} where $\ast$ denote any continuous t-norm.

We denote by $A^\cdot(X)$) the space of all $\cdot$ -measures considered as a subspace of  the space $[0,1]^{C(X,[0,1])}$ with the
product topology.
For a map $\phi\in C(X,[0,1])$ we denote by $\pi_\phi$ or $\pi(\phi)$ the
corresponding projection $\pi_\phi:A^\cdot X\to [0,1]$. For each continuous map $f:X\to Y$ between compacta $X$ and $Y$ we define the
map $A^\cdot f:A^\cdot X\to A^\cdot Y$ by the formula $\pi_\phi\circ A^\cdot f=\pi_{\phi\circ f}$ for $\phi\in C(Y,[0,1])$. It is easy
to check that the map $A^\cdot f$ is well defined and continuous. The construction $A^\cdot$ forms an endofunctor on $\Comp$ (see
\cite{Sukh} for more details).
For a compactum $X$ we define components $hX$ and $mX$ of natural transformations $h:\Id_{\Comp}\to A^\cdot$, $m:(A^\cdot)^2\to
A^\cdot$ by $\pi_\phi\circ hX=\phi$ and $\pi_\phi\circ m X=\pi(\pi_\phi)$ for all $\phi\in C(X,[0,1])$). It from the results of
\cite{R2} that the triple $\A^\cdot=(A^\cdot,h,m)$ is a monad on $\Comp$.  Let us remark that for each $x\in X$ the $\cdot$ -measure
$hX(x)$ is the Dirac measure concentrated at the point $x$ and we denote   $hX(x)=\delta_x$.

The main goal of this paper is to prove the isomorphism of the monads introduced above using the densities of idempotent measures.
The notion of density for an idempotent measure was introduced in \cite{A}. We will consider the correspondence between the measures
discussed here and their respective densities.

Let $\R_{\max}=\R\cup\{-\infty\}$ be the metric space endowed with the metric $\varrho$ defined by $\varrho(x, y) = |e^x-e^y|$.
 The convention $-\infty + x=-\infty$ and $-\infty \vee x=x$ allows us to extend the operations $\vee$ and $+$  over $\R_{\max}$. By
 $[-\infty,0]$ we denote the corresponding subset of $\R_{\max}$.

Let us recall that a map $f:X\to [-\infty,0]$ is called  upper semi continuous if $f^{-1}([-\infty,t])$ is open for each
$t\in[-\infty,0]$.

Let $X$ be a compactum. Put $$DX=\{f:X\to [-\infty,0]\mid f \text{ is upper semi continuous and }\max f = 0\}.$$

Let us remark that existence of max in the above notation follows from the upper semi continuity  of $f$.

Let $A\subset X$ and $t\in[-\infty,0]$. We denote $$S_-(A,t)=\{f\in DX\mid f(a)<t \text{ for each }a\in A\}$$  and $$S_+(A,t)=\{f\in
DX\mid \text{ there exists }a\in A \text{ such that } f(a)>t \}.$$  We consider $DX$ with the topology generated by a  subbase
$$\C=\{S_-(A,t)\mid A \text{ is a closed subset of } X \text{ and }t\in[-\infty,0]\}\cup$$
$$\cup\{S_+(U,t)\mid U \text{ is an open subset of } X \text{ and }t\in[-\infty,0]\}.$$

Define  maps $nX:DX\to IX$ and $sX:IX\to DX$ as follows $$nX(f)(\varphi) = \max\{f(x)+\varphi(x) | x \in X\}$$ for $f\in DX$,
$\varphi\in C(X)$ and $$sX(\mu)(x)=\inf\{\mu(\varphi)|\varphi\in C(X,(-\infty,0])\text{ such that  }\varphi(x)=0\}$$ for $\mu\in IX$,
$x\in X$.

\begin{proposition}\label{maps} The maps $nX$ and $sX$ are well-defined and continuous. Moreover, we have $nX\circ sX=\id_{IX}$ and
$sX\circ nX=\id_{DX}$.
\end{proposition}

\begin{proof} Consider any $f\in DX$ and $x\in X$. Then we have $$sX\circ nX(f)(x)=\inf\{nX(f)(\varphi)|\varphi\in C(X,(-\infty,0])\text{ such that
}\varphi(x)=0\}=$$ $$=\inf\{\max\{f(y)+\varphi(y) | y \in X\}|\varphi\in C(X,(-\infty,0])\text{ such that  }\varphi(x)=0\}=f(y).$$
Hence $sX\circ nX=\id_{DX}$.

Since $I_\omega X$   is dense in $IX$, it is enough to check the equality $nX\circ sX=\id_{IX}$ for
$\nu=\max_{i=1}^n(\lambda_i+\delta_{x_i})$. Consider any $\varphi\in C(X)$. Then $\nu(\varphi)=\max_{i=1}^n(\lambda_i+\varphi(x_i))$.
We have $$nX\circ sX(\nu)(\varphi)=\max\{sX(\nu)(x)+\varphi(x) | x \in X\}=$$ $$=\max\{\varphi(x)+\inf\{\nu(\varphi)|\varphi\in
C(X,(-\infty,0])\text{ such that  }\varphi(x)=0\} | x \in X\}=$$ $$=\max_{i=1}^n(\lambda_i+\varphi(x_i))=\nu(\varphi).$$
\end{proof}

Let us remark that the last proposition implies compactness of $DX$.

Analogously we introduce the space of densities of $\cdot$ -measures.   Let $X$ be a compactum. Put $$D_1 X=\{f:X\to [0,1]\mid f
\text{ is upper semicontinuous and }\max f = 1\}.$$
Let $\mu\in A^\cdot(X)$. We denote $$S_-(B,t)=\{f\in DX\mid f(a)<t \text{ for each }a\in A\}$$ where $B$ is a closed subset of $X$,
$t\in[0,1]$ and $$S_+(U,t)=\{f\in DX\mid \text{ there exists }a\in U \text{ such that } f(a)>t \}$$ where $U$ is an open subset of
$X$, $t\in[0,1]$. We consider $D_1X$ with the topology generated by a  subbase
$$\C=\{S_-(B,t)\mid B \text{ is a closed subset of } X \text{ and }t\in[0,1]\}\cup$$
$$\cup\{S_+(U,t)\mid U \text{ is an open subset of } X \text{ and }t\in[0,1]\}.$$

Define  maps $n_1X:D_1X\to A^\cdot(X)$ and $s_1X:A^\cdot(X)\to D_1X$ as follows $$nX(f)(\varphi) = \max\{f(x)\cdot\varphi(x) | x \in
X\}$$ for $f\in DX$, $\varphi\in C(X,[0,1])$ and $$s_1X(\mu)(x)=\inf\{\mu(\varphi)|\varphi\in C(X,[0,1])\text{ such that
}\varphi(x)=1\}$$ for $\mu\in A^\cdot(X)$, $x\in X$.

The proof of the following proposition is analogous to the proof of Proposition  \ref{maps}.

\begin{proposition}\label{maps1} The maps $n_1X$ and $s_1X$ are well-defined and continuous. Moreover, we have $n_1X\circ
s_1X=\id_{A^\cdot(X)}$ and $s_1X\circ n_1X=\id_{D_1X}$.
\end{proposition}

The last proposition implies compactness of $D_1X$.

\section{Categorical aspects of correspondence between idempotent measures and its densities}

We will show that introduced homeomorphism between spaces $IX$ and $DX$ has categorical character. We introduce in this section a
monad structure on the construction $DX$ and show that this monad is isomorphic to the idempotent measure monad. For technical reasons
we firstly introduce  for the construction a weaker structure of quasimonad.  We will need some general  notions concerning semimonads
and its isomorphisms.

A {\it quasimonad} $\E=(E,\eta,\mu)$ in the category $\Comp$ consists of an endofunctor $E:{\Comp}\to{\Comp}$ and natural
transformations $\eta:\Id_{\Comp}\to E$ (unity), $\mu:E^2\to E$ (multiplication) satisfying the relations $$\mu\circ
E\eta=\mu\circ\eta E=\text{\bf 1}_E.$$
The notion of quasimonad  was introduced in \cite{FF} under name semimonad. But the term semimonad was used in many another sources as
monad without unity, so, we prefer use the term quasimonad to denote monad without associativity of multiplication.

A natural transformation  $\psi:E\to E'$ of functors $E$ and $E'$ is called a {\it morphism}
from a (quasi)monad $\E=(E,\eta,\mu)$ into a (quasi)monad $\E'=(E',\eta',\mu')$
if $\psi\circ\eta= \eta'$ and $\psi\circ\mu=\mu'\circ\psi E'\circ
E\psi$. A (quasi)monad morphism $\psi:E\to E'$ is called an isomorphism if it has an inverse morphism of (quasi)monads. It is easy to
check that in $\Comp$ a (quasi)monad morphism $\psi$ is an isomorphism   if each its  component $\psi X$ is a homeomorphism.

\begin{proposition}\label{monad} Let $\psi:\E\to \E'$ be an isomorphism of quasimonads $\E=(E,\eta,\mu)$ and  $\E'=(E',\eta',\mu')$.
If $\E'$ is a monad then $\E$ is a monad too.
\end{proposition}

\begin{proof} We have to prove the equality $\mu\circ\mu E=\mu\circ E\mu.$ Let $\varphi:\E'\to \E$ be the morphism inverse to $\psi$.
The equalities $\mu'\circ\mu' E=\mu'\circ E\mu'$, $\psi\circ\mu=\mu'\circ\psi E'\circ E\psi$ and naturality of all transformations
imply the following equalities for any compactum $X$.
$$\mu X\circ E(\mu X)=\varphi X\circ\mu' X\circ\psi E'X\circ E(\psi X)\circ E(\mu X)=$$
$$=\varphi X\circ\mu' X\circ\psi E'X\circ E(\mu' X)\circ E(\psi E'X)\circ E^2(\psi X)=$$
$$=\varphi X\circ\mu' X\circ E'(\mu' X)\circ\psi E'^2X\circ  E(\psi E'X)\circ E^2(\psi X)=$$
$$=\varphi X\circ\mu' X\circ \mu' E'X\circ\psi E'^2X\circ  E(E'(\psi X))\circ  E(\psi EX) =$$
$$=\varphi X\circ\mu' X\circ \mu' E'X\circ E'^2(\psi X)\circ\psi  E'(E( X))\circ  E(\psi EX) =$$
$$=\varphi X\circ\mu' X\circ  E'(\psi X)\circ \mu'EX\circ\psi  E'(E( X))\circ  E(\psi EX) =$$
$$=\varphi X\circ\mu' X\circ  E'(\psi X)\circ \psi  EX\circ  \mu EX =\mu X\circ \mu EX.$$
\end{proof}

For each continuous map $g:X\to Y$ between compacta $X$ and $Y$ we define the map $D g:D X\to D Y$ by the formula $D g(f)(y)=max
f(g^{-1}(y))$ for $f\in DX$ and $y\in Y$ (we put $max \emptyset=-\infty$).

\begin{proposition}\label{functor} The map $D g$ is well-defined and continuous. Moreover, we have $D(\id_X)=\id_{DX}$ and $D(g\circ
h)= D(g)\circ D(h)$.
\end{proposition}

The last proposition implies that  $D$ forms an endofunctor on $\Comp$.
For a compactum $X$ we define components $\varepsilon X$ and $\kappa X$ of natural transformations $\varepsilon:\Id_{\Comp}\to D$,
$\kappa:D^2\to D$ by $$
\varepsilon X(x)(y)=\begin{cases}
0,&x=y,\\
-\infty,&x\neq y\end{cases}
$$
for $y$, $x\in X$ and
$$\kappa X(F)(x)=\max\{f(x)+F(f)\mid f\in DX\}$$ for  $F\in D^2 X$ and $x\in X$.

\begin{proposition}\label{quasimonad} The triple $\D=(D,\varepsilon,\kappa)$ is a quasimonad on $\Comp$.
\end{proposition}

\begin{proof} Consider any compactum $X$ and a point $x\in X$.

We have $$\kappa X(D(\varepsilon X)(f))(x)=\max\{g(x)+D(\varepsilon X)(f)(g)\mid g\in DX  \}=$$ $$=\max\{g(x)+\max f((\varepsilon
X)^{-1}(g))\mid g\in DX  \}=f(x)$$
 for each $f\in DX$. Thus $\kappa X\circ D(\varepsilon X)=\id_{DX}$.

We have $$\kappa X(\varepsilon DX(f))(x)=\max\{g(x)+\varepsilon DX(f)(g)\mid g\in DX  \}=f(x)$$ for each $f\in DX$. Thus $\kappa
X\circ \varepsilon DX=\id_{DX}$.
\end{proof}

It is a routine checking that the maps $sX$ are components of the natural transformation $s:I\to D$.

\begin{theorem}\label{isomorph} The natural transformation $s$ is an isomorphism of quasimonads $\I$ and $\D$.
\end{theorem}

\begin{proof} Consider any compactum $X$. We have to check equalities $sX\circ\eta X= \varepsilon X$ and $sX\circ\mu X=\kappa X\circ
sDX\circ
I(sX)$.

Take any $x\in X$. Then we have $$sX(\eta X(x))(y)=\inf\{\eta X(x)(\psi)\mid \psi\in C(X) \text{ such that } \psi(x)\ge 0\}=$$
$$=\begin{cases}
0,&x=y\\
-\infty,&x\neq y\end{cases}=\varepsilon X(x)(y)$$ for each $y\in X$.

Now, consider any $N\in I^2 X$ and $x\in X$. Then we have
$$\kappa X(sDX(I(sX)(N)))(x)=\max\{f(x)+sDX(I(sX)(N))(f)\mid f\in DX\}=$$
$$=\max\{f(x)+\inf\{N(\Phi\circ sX)\mid \Phi\in C(DX) \text{ such that } \Phi(f)\ge 0\}\mid f\in DX\}=a$$
and
$$sX(\mu X(N))(x)=\inf\{\mu(N)(\psi)\mid \psi\in C(X) \text{ such that } \psi(x)\ge 0\}=$$ $$=\inf\{N(\pi_\psi)\mid \psi\in C(X)
\text{ such that } \psi(x)\ge 0\}=b.$$

For $f\in DX$ and $\psi\in C(X) \text{ such that } \psi(x)\ge 0$ put $\Phi=(\pi_\psi+c_{-f(x)})\circ nX$. Then
$\Phi(f)=\max\{f(y)+\psi(y)\mid y\in X\}+(-f(x))\ge f(x)+\psi(x)-f(x)\ge 0$. On the other hand $f(x)+N(\Phi\circ
sX)=f(x)+N(\pi_\psi+c_{-f(x)})=N(\pi_\psi)$. Hence $a\le b$.

Since $I_\omega X$ is dense in $IX$, it is enough to prove the inverse inequality for $N\in I_\omega (I_\omega X)$. Let
$N=\max_{i=1}^n(\lambda_i+\delta_{\mu_i})$ where $\mu_i=\max_{s=1}^{k_i}(\lambda_{is}+\delta_{x_{is}})\in I_\omega X$.

If $x\neq x_{is}$ for each $i\in\{1,\dots,n\}$ and $s\in\{1,\dots,k_i\}$, then $b=-\infty$. Consider the case when $x= x_{is}$ for
some $i\in\{1,\dots,n\}$ and $s\in\{1,\dots,k_i\}$. Choose $j\in\{1,\dots,n\}$ such that $x= x_{js}$ for  some $s\in\{1,\dots,k_j\}$
and $\lambda_j+\lambda_{js}=\max\{\lambda_i+\lambda_{is}\mid i\in\{1,\dots,n\}$ and $s\in\{1,\dots,k_i\}$ such that $x= x_{is}\}$.

Consider $f\in DX$ defined as follows $$
f(y)=\begin{cases}
\lambda_{jl},&y= x_{jl} \text{ for some } l\in\{1,\dots,k_j\},\\
-\infty,& y\neq x_{jl} \text{ for each } l\in\{1,\dots,k_j\}.\end{cases}
$$

Then $sDX(IsX)(N)(f)=\lambda_j$ and $f(x)=\lambda_{js}$. It is easy to see that $a=\lambda_j+\lambda_{js}$.

Now, consider $\psi\in C(X)$ such that $\psi(x)=0$ and $\psi(x_{jl})=-\infty$ for each $l\in\{1,\dots,k_j\}$ such that $x\neq x_{jl}$.
Then $N(\pi_\psi)=\lambda_j+\lambda_{js}$, hence $b\le a$.

\end{proof}

Since the quasimonad $\I$ is a monad \cite{Zar}, Proposition \ref{monad} and Theorem \ref{isomorph} imply the following theorem.

\begin{theorem} The triple $\D=(D,\varepsilon,\kappa)$ is a monad on $\Comp$ isomorphic to the monad $\I$.
\end{theorem}

We also can complete the construction $D_1$ to a monad. For each continuous map $g:X\to Y$ between compacta $X$ and $Y$ we define the
map $D_1 g:D_1 X\to D_1 Y$ by the formula $D_1 g(f)(y)=max f(g^{-1}(y))$ for $f\in D_1X$ and $y\in Y$ (here we put $max
\emptyset=0$).
It is easy to check that  $D_1$ forms an endofunctor on $\Comp$.

For a compactum $X$ we define components $\varepsilon_1 X$ and $\kappa_1 X$ of natural transformations $\varepsilon_1:\Id_{\Comp}\to
D_1$, $\kappa_1:D_1^2\to D_1$ by $$
\varepsilon_1 X(x)(y)=\begin{cases}
1,&x=y,\\
0,&x\neq y\end{cases}
$$
for $y$, $x\in X$ and
$$\kappa_1 X(F)(x)=\max\{f(x)\cdot F(f)\mid f\in D_1X\}$$ for  $F\in D_1^2 X$ and $x\in X$.

The following theorem is derived from the analogous arguments employed in this section.

\begin{theorem} The triple $\D_1=(D_1,\varepsilon_1,\kappa_1)$ is a monad on $\Comp$  and the natural transformation $n_1:D_1\to
A^\cdot$, $s_1:A^\cdot\to D_1$ are monad isomorphisms.
\end{theorem}

\section{The main result and some consequences} Now it is enough to prove isomorphism of densities monads in order to obtain that
monads $\I$ and $\A^\cdot$ are isomorphic. For a compactum $X$ define the map $lX:DX\to D_1X$ as follows $lX(f)(x)=\exp(f(x))$ for
$f\in DX$ (we put $\exp(-\infty)=0$). It is easy to check that $lX$ is well-defined and continuous and the maps $lX$ are components of
the natural transformation $l:D\to D_1$.

\begin{theorem}\label{isomorph D} The natural transformation $l$ is a morphism of monads $\D$ and  $\D_1$.
\end{theorem}

\begin{proof} Consider any compactum $X$. The map $lX$ has inverse continuous map $pX:D_1X\to DX$ defined as follows
$pX(g)(x)=\ln(g(x))$ for $g\in D_1X$ (we put $\ln(0)=-\infty$). The equality $lX\circ\varepsilon X= \varepsilon_1 X$ is obvious.   We
have to check the equality  and $lX\circ\kappa X=\kappa_1 X\circ lDX\circ
D_1(lX)$.

Now, consider any $F\in D^2 X$ and $x\in X$. Then we have
$$lX(\kappa X(F))(x)=\exp(\kappa X(F)(x))=\exp(\max\{g(x)+F(g)\mid g\in DX\})=$$
$$=\max\{\exp(g(x))\cdot\exp(F(g))\mid g\in DX\}=\max\{f(x)\cdot \exp(F(pX(f)))\mid f\in D_1X\}=$$
$$=\max\{f(x)\cdot \exp((D_1(lX)(F))(f)\mid f\in D_1X\}=$$
$$=\max\{f(x)\cdot lDX(D_1(lX)(F))(f)\mid f\in D_1X\}=\kappa_1 X(lDX(D_1(lX)(F)))(x).$$
\end{proof}

Combining Theorem \ref{isomorph D} with results of the previous section we obtain the main result of our paper.

\begin{theorem}\label{main} The  monads $\A^\cdot$ and  $\I$ are isomorphic.
\end{theorem}

An isomorphism between the monad $\A^\cdot$ and a monad related to possibility capacities was established in  \cite{R2}. Let us remind
shortly some notions and facts related to that isomorphism. We start from the definition of capacity on a compactum $X$. We follow a
terminology of \cite{NZ}. By $\F(X)$ we denote the family of all closed subsets of a compactum $X$.

\begin{df} A function $\nu:\F(X)\to [0,1]$  is called an {\it upper-semicontinuous capacity} on $X$ if the three following properties
hold for each closed subsets $F$ and $G$ of $X$:

1. $\nu(X)=1$, $\nu(\emptyset)=0$,

2. if $F\subset G$, then $\nu(F)\le \nu(G)$,

3. if $\nu(F)<a$ for $a\in[0,1]$, then there exists an open set $O\supset F$ such that $\nu(B)<a$ for each compactum $B\subset O$.
\end{df}

A capacity $\nu$ is extended in \cite{NZ} to all open subsets $U\subset X$ by the formula $$\nu(U)=\sup\{\nu(K)\mid K \text{ is a
closed subset of }  X \text{ such that } K\subset U\}.$$

It was proved in \cite{NZ} that the space $MX$ of all upper-semicontinuous  capacities on a compactum $X$ is a compactum as well, if a
topology on $MX$ is defined by a subbase that consists of all sets of the form $O_-(F,a)=\{c\in MX\mid c(F)<a\}$, where $F$ is a
closed subset of $X$, $a\in [0,1]$, and $O_+(U,a)=\{c\in MX\mid c(U)>a\}$, where $U$ is an open subset of $X$, $a\in [0,1]$. Since all
capacities we consider here are upper-semicontinuous, in the following we call elements of $MX$ simply capacities.

A capacity $\nu\in MX$ for a compactum $X$ is called  a possibility capacity if for each family $\{A_t\}_{t\in T}$ of closed subsets
of $X$ such that $\bigcup_{t\in T}A_t$ is a closed subset of $X$ we have  $\nu(\bigcup_{t\in T}A_t)=\sup_{t\in T}\nu(A_t).$
(See \cite{WK} for more details.)
We denote by $\Pi X$ a subspace of $MX$ consisting of all possibility capacities. Since $X$ is compact and $\nu$ is
upper-semicontinuous, $\nu\in \Pi X$ iff $\nu$ satisfies the simpler requirement that $\nu(A\cup B)=\max\{\nu(A),\nu(B)\}$ for each
closed subsets $A$ and $B$ of $X$.  It is easy to check that  $\Pi X$ is a closed subset of $MX$.

The construction $\Pi$ was completed to the monad $\U_\cdot=(\Pi,\eta,\mu_\cdot)$ (where $\cdot$ is the usual multiplication
operation) in \cite{NR}.  For a continuous map of compacta $f:X\to Y$ we define the map $f:\Pi X\to \Pi Y$ by the formula $\Pi
f(\nu)(A)=\nu(f^{-1}(A))$ where $\nu\in \Pi X$ and $A$ is a closed subset of $Y$. The map $\Pi f$ is continuous.  In fact, this
extension of the construction $\Pi$ defines the possibility capacity functor $\Pi$ in the category $\Comp$.

The components of the  natural transformations $\eta$ and $\mu_\cdot$ are defined as follows:
$$
\eta X(x)(F)=\begin{cases}
1,&x\in F,\\
0,&x\notin F;\end{cases}
$$

For a closed set $F\subset X$ and for $t\in [0,1]$ put $F_t=\{c\in MX\mid c(F)\ge t\}$. Define the map $\mu_\cdot X:\Pi^2 X\to \Pi X$
by the formula $$\mu_\cdot X(\C)(F)=\max\{\C(F_t)\cdot t\mid t\in(0,1]\}.$$

An isomorphism of monads $\U_\cdot$ and $A^\cdot$  was built in \cite{R6}.  Composing this isomorphism with the isomorphism between
monads $A^\cdot$ and $\I$ built in this paper we obtain an isomorphism between monads $\U_\cdot$ and  $\I$. After some technical
transformations  we can represent components of this isomorphism by the formula $iX(c)(\varphi)=\max\{\varphi(x)+\ln(c(\{x\}))\mid
x\in X\}$ for a compactum $X$, $c\in \Pi(X)$ and $\varphi\in C(X)$. We can consider $iX(c)(\varphi)$ as an fuzzy integral of the
function $\varphi$ related to the
possibility capacity $c$. Evidently we have the following characterization of such integral: $\nu\in IX$ iff there exists $c\in\Pi X$
such that $\nu(\varphi)=iX(c)(\varphi)$ for any $\varphi\in C(X)$.

Certainly the above formula has sense only for possibility capacity. But we can consider another representation. We denote
$\varphi_t=\{x\in X\mid \varphi(x)\ge t\}$ for $\varphi\in C(X)$ and $t\in\R$.

\begin{proposition}\label{repr} We have $\max\{\varphi(x)+\ln(c(\{x\}))\mid x\in X\}=\max\{\ln(c(\varphi_t))+ t\mid t\in\R\}$ for any
$c\in \Pi(X)$ and $\varphi\in C(X)$.
\end{proposition}

Let us remark that the second formula can by applied to any capacity. We use it to define a fuzzy integral. Let $X$ be a compactum.

\begin{df} Let $f\in  C(X)$ be a function and $c\in MX$. The max-plus integral of $f$ w.r.t. $c$ is given by  the formula
$$\int_X^{\vee+} fdc=\max\{\ln(c(\varphi_t))+ t\mid t\in\R\}.$$
\end{df}

Let us remark that we can consider the defined above integral as a logarithmic version of the Shilkret integral \cite{Shi}. Now, we
will obtain some characterization of the max-plus integral.

Let $X$ be a compactum.  We call two functions $\varphi$, $\psi\in C(X,[0,1])$ comonotone (or equiordered) if
$(\varphi(x_1)-\varphi(x_2))\cdot(\psi(x_1)-\psi(x_2))\ge 0$ for each $x_1$, $x_2\in X$. Let us remark that a constant function is
comonotone to any function $\psi\in C(X,[0,1])$.

\begin{theorem}\label{charac} A functional $I:C(X)\to\R$ satisfies the conditions:
\begin{enumerate}
\item  $I(1_X)=1$;
\item  $I(\psi\vee\varphi)=I(\psi)\vee I(\varphi)$ for each comonotone functions $\varphi$, $\psi\in C(X,[0,1])$;
\item  $I(\lambda_X+\varphi)=\lambda+I(\varphi)$ for each $\lambda\in \R$ and $\varphi\in C(X)$

\end{enumerate}

if and only if there exists $c\in MX$ such that $I(\varphi)=\int_X^{\vee+} \varphi dc$ for each $\psi\in C(X,[0,1])$.

\end{theorem}

 \section{Convexities  related to monads $\I$ and $\A^\cdot$}

 Max-plus convex sets were introduced in \cite{Z} and found many applications (see for example \cite{BCh}). It was shown in \cite{R1}
 that each monad generates a convexity structure on its algebras. We will show in this section that convexities generated by the monad
 $\I$ coincide with  max-plus convexities.

Let $\tau$ be a cardinal number. Given $x, y \in \R^\tau$ and $\lambda\in[-\infty,0]$, we denote by $y\vee x$ the coordinatewise
maximum of x and y and by $\lambda+ x$ the vector obtained from $x$ by adding $\lambda$ to each of its coordinates. A subset $A$ in
$\R^\tau$ is said to be  max-plus convex if $(\alpha+ a)\oplus  b\in A$ for all $a, b\in A$ and $\alpha\in[-\infty,0]$. It is easy to
check that $A$  is   max-plus convex iff $\vee_{i=1}^n\lambda_i+ x_i\in A$ for all $x_1,\dots, x_n\in A$ and
$\lambda_1,\dots,\lambda_n\in[-\infty,0]$ such that $\vee_{i=1}^n\lambda_i=0$. In the following by max-plus convex compactum we mean a
max-plus convex compact subset of $\R^\tau$. It is shown in \cite{Zar} that $IX$ is a compact max-plus subset of $\R^{C(X)}$.

We will need some categorical notions and the construction of convexities generated by a monad from \cite{R1}.  Let $\T=(T,\eta,\mu)$
be a monad in the category ${\Comp}$. A
pair $(X,\xi)$, where $\xi:TX\to X$ is a map, is called a $\T$-{\it
algebra} if $\xi\circ\eta X=id_X$ and $\xi\circ\mu X=\xi\circ
T\xi$. Let $(X,\xi)$, $(Y,\xi')$ be two $\T$-algebras. A map
$f:X\to Y$ is called a  morphism of $\T$-algebras if $\xi'\circ
Tf=f\circ\xi$. A morphism of $\T$-algebras
 is called an isomorphism   if there exists an inverse morphism of $\T$-algebras.

Let $(X,\xi)$ be an $\F$-algebra for a monad $\F=(F,\eta,\mu)$ and let $A$ be a closed subset of $X$. Denote by $\chi_A$ the quotient
map
$\chi_A:X\to X/A$ (the equivalence classes  are one-point sets $\{x\}$ for $x\in X\setminus A$ and the set $A$) and put $a=\chi_A(A)$.
Denote $A^+=(F\chi_A)^{-1}(\eta(X/A)(a))$.    Define the $\F$-{\it convex
hull} $\conv_\F(A)$ of $A$ as follows
$\conv_\F(A)=\xi(A^+)$. Put additionally
$\conv_\F(\emptyset)=\emptyset$. We define the family
$\C_\F(X,\xi)=\{A\subset X|A $ is closed and $\conv_\F(A)=A\}$.
The elements of the family $\C_\F(X,\xi)$ will be called $\F$-{\it convex}.

Let $A\subset  \R^T$ be a compact max-plus convex subset. For each $t\in T$ we put $f_t=\pr_t|_A:A\to \R$ where $\pr_t:\R^T\to\R$ is
the natural projection.    Given $\mu\in IA$, the point $\beta_A(\mu)\in\R^T$ is defined by the conditions
$\pr_t(\beta_A(\mu))=\mu(f_t)$ for each $t\in T$. It is shown in \cite{Zar} that $\beta_A(\mu)\in  A$ for each $\mu\in I(A)$ and the
map $\beta_A : I(A)\to A$ is continuous.
The map $\beta_A$ is called the idempotent barycenter map \cite{Zar}. The following proposition can be proved by a routine checking.

\begin{proposition} Let $K\subset  \R^\tau$ be a compact max-plus convex convex subset. Then the pair $(K,\beta_K)$ is an
$\I$-algebra.
\end{proposition}

It is natural to ask whether  each $\I$-algebra has the above described form? More precisely, we have the following problem.

\begin{problem} Let $(X,\xi)$ be an $\I$-algebra. Is $(X,\xi)$ isomorphic to $(K,\beta_K)$ for some max-plus convex compactum
$K\subset  \R^T$?
\end{problem}

For a max-plus convex compactum $K\subset  \R^T$ we denote the family of $\I$-convex subset of $K$ by $\C_{\I}(K,\beta_K)$ and the
family of max-plus convex subsets by $\C_{\vee+}(K)$.

The proof of the following theorem is analogous to the proof of Theorem 3  from \cite{R2}.

\begin{theorem} Let $K$ be a max-plus convex compactum. Then  $\C_{\I}(K,\beta_K)=\C_{\vee+}(K)$.
\end{theorem}

Let us remark that Theorem 3  from \cite{R2} established equality between the convexity generated by the monad $\A^\cdot$ and
convexity related to the maximum and the multiplication operation in $[0,1]^\tau$. The isomorphism of monads $\I$ and $\A^\cdot$
implies isomorphism of corresponding convexities. It is worth to notice that an isomorphism between max-plus convexity in
$[-\infty,+\infty)^n$ and max- $\cdot$ convexity in $[0,+\infty)^n$ for natural $n$ was mentioned in \cite{BCh}.

\end{document}